\documentclass{amsart}
\usepackage{amsmath,amsthm,amssymb}

\newtheorem{defn}{Definition}[section]
\newtheorem{Prop}[defn]{Proposition}
\newtheorem{Coro}[defn]{Corollary}

\newtheorem{clm}[defn]{Claim}
\newtheorem{thm}[defn]{Theorem}
\newtheorem{lem}[defn]{Lemma}
\theoremstyle{definition}

\newcommand{\N}{\mathbb{N}}
\newcommand{\coideal}{\mathcal{H}}
\def\mah{\mathbb M_{\mathcal H}}
\def\col{Col(\omega , <\lambda)}

\begin{document}

\title{Ideal games and Ramsey sets}

\author{Carlos Di Prisco}
\address{Instituto Venezolano de Investigaciones Cient\'ificas y
Escuela de Matem\'atica, Universidad Central de Venezuela,
Caracas, Venezuela}

\email{cdiprisc@ivic.gob.ve}

\author{Jos\'e G. Mijares}
\address{Instituto Venezolano de Investigaciones Cient\'ificas y
Escuela de Matem\'atica, Universidad Central de Venezuela,
Caracas, Venezuela}

\email{jmijares@ivic.gob.ve}

\author{Carlos Uzc\'ategui}
\address{Departamento de Matem\'aticas, Facultad de Ciencias, Universidad de Los Andes, M\'erida, Venezuela}

\email{uzca@ula.ve}



\begin{abstract}
It is shown that Matet's characterization (\cite{mate}) of the
Ramsey property relative to a selective co-ideal $\mathcal{H}$, in
terms of games of Kastanas (\cite{kas}), still holds if we
consider semiselectivity (\cite{far}) instead of selectivity.
Moreover, we prove that a co-ideal $\mathcal{H}$ is semiselective
if and only if Matet's game-theoretic characterization of the
$\mathcal{H}$-Ramsey property holds. This lifts Kastanas's
characterization of the classical Ramsey property to its optimal
setting, from the point of view of the local Ramsey theory and
gives a game-theoretic counterpart to a theorem of Farah
\cite{far}, asserting that a co-ideal $\mathcal{H}$ is
semiselective if and only if the family of $\mathcal{H}$-Ramsey
subsets of $\N^{[\infty]}$ coincides with the family of those sets
having the abstract $\mathcal{H}$-Baire property. Finally, we show
that under suitable assumptions, for every semiselective co-ideal
$\mathcal H$ all sets of real numbers are $\mathcal H$-Ramsey.
\end{abstract}

\maketitle

\section{Introduction}

Let $\N$ be the set of nonnegative integers. Given an infinite set
$A\subseteq \N$, the symbol $A^{[\infty]}$ (resp. $A^{[<\infty]}$)
represents the collection of the infinite (resp. finite) subsets
of A. Let $A^{[n]}$ denote the set of all the subsets of A with n
elements. If $a\in \N^{[<\infty]}$ is an \textbf{initial segment}
of $A\in \N^{[\infty]}$ then we write $a\sqsubset A$. Also, let
$A/a : = \{n\in A : max(a) < n\}$, and write $A/n$ to mean
$A/\{n\}$. For $a\in \N^{[<\infty]}$ and $A\in\N^{[\infty]}$ let
\[
[a,A] : = \{B\in\N^{[\infty]} : a\sqsubset B\subseteq A\}.
\]

The family $\{[a,A]: (a,A)\in \N^{[<\infty]}\times
\N^{[\infty]}\}$ is a basis for \textbf{Ellentuck's topology},
also known as \textbf{exponential topology}. In \cite{ell},
Ellentuck gave a characterization of the Ramsey property in terms
of the Baire property relative to this topology (see Theorem
\ref{Ellentuck} below).

\medskip

Let $(P,\leq)$ be a poset, a subset $D\subseteq P$ is
\textbf{dense} in $P$ if for every $p \in P$, there is $q \in D$
with $q\leq p$. $D\subseteq P$ is \textbf{open} if $p\in D$ and
$q\leq p$ imply $q\in D$. $P$ is $\sigma$-\textbf{distributive} if
the intersection of countably many dense open subsets of $P$ is
dense. $P$ is $\sigma$-\textbf{closed} if every decreasing
sequence of elements of $P$ has a lower bound.

\medskip

\begin{defn}\label{co-ideal}
A family $\mathcal{H}\subset\wp (\N)$ is a \textbf{co-ideal} if it satisfies:
\begin{itemize}
\item[{(i)}] $A\subseteq B$ and $A\in\mathcal{H}$ implies $B\in\mathcal{H}$, and
\item[{(ii)}] $A\cup B \in \mathcal{H}$ implies $A\in\mathcal{H}$ or $B\in\mathcal{H}$.
\end{itemize}
\end{defn}

The complement $\mathcal{I} = \wp(\N)\setminus\mathcal{H}$ is the
\textbf{dual ideal} of $\mathcal{H}$. We will suppose that
co-ideals differ from $\wp (\N)$. Also, we say that a nonempty
family $\mathcal{F}\subseteq\mathcal{H}$ is
$\mathcal{H}$-\textbf{disjoint} if for every $A,B\in\mathcal{F}$,
$A\cap B\not\in\mathcal{H}$. We say that $\mathcal{F}$ is a
\textbf{maximal} $\mathcal{H}$-\textbf{disjoint family} if it is
$\mathcal{H}$-disjoint and it is not properly contained in any
other $\mathcal{H}$-disjoint family.
%as a subset.

\bigskip

A subset $\mathcal X$ of $\N^{[\infty]}$ is \textbf{Ramsey} if for
every $[a,A]\neq\emptyset$ with $A\in\N^{[\infty]}$ there exists
$B\in [a,A]$ such that $[a,B]\subseteq\mathcal{X}$ or
$[a,B]\cap\mathcal{X} = \emptyset$. Some authors have used the
term ``completely Ramsey'' to express this property, reserving the
term ``Ramsey'' for a weaker property. Galvin and Prikry
\cite{gapr} showed that all Borel subsets of $\N^{[\infty]}$ are
Ramsey, and Silver \cite{sil} extended this to all analytic sets.
Mathias in \cite{math} showed that if the existence of an
inaccessible cardinal is consistent with $ZFC$ then it is
consistent, with $ZF+DC$, that every subset of $\N^{[\infty]}$ is
Ramsey. Mathias introduced the concept of a selective co-ideal (or
a happy family), which has turned out to be of wide interest.
Ellentuck \cite{ell} characterized the Ramsey sets as those having
the Baire property with respect to the exponential topology of
$\N^{[\infty]}$.

A game-theoretic characterization of the Ramsey property was given by
Kastanas in \cite{kas}, using games in the style of Banach-Mazur
with respect to Ellentuck's topology.

\bigskip

In this work we will deal with a game-theoretic characterization
of the following property:

\begin{defn}
Let $\mathcal{H}\subset\N^{[\infty]}$ be a co-ideal.
$\mathcal{X}\subseteq\N^{[\infty]}$ is
$\mathcal{H}$-\textbf{Ramsey} if for every $[a,A]\neq\emptyset$
with $A\in\mathcal{H}$ there exists $B\in [a,A]\cap\mathcal{H}$
such that $[a,B]\subseteq\mathcal{X}$ or $[a,B]\cap\mathcal{X} =
\emptyset$. $\mathcal{X}$ is $\mathcal{H}$-\textbf{Ramsey null} if
for every $[a,A]\neq\emptyset$ with $A\in\mathcal{H}$ there exists
$B\in [a,A]\cap\mathcal{H}$ such that $[a,B]\cap\mathcal{X} =
\emptyset$.
\end{defn}

Mathias considered sets that are $\mathcal H$-Ramsey with respect
to a selective co-ideal $\mathcal H$, and generalized Silver's
result to this context. Matet \cite{mate} used games to
characterize sets which are  Ramsey with respect to a selective
co-ideal $\mathcal H$. These games coincide with the games of
Kastanas if $\mathcal H$ is $\N^{[\infty]}$ and with the games of
Louveau \cite{lou} if $\mathcal H$ is a Ramsey ultrafilter.

Given a co-ideal $\mathcal{H}\subset\N^{[\infty]}$, the collection
$\{[a,A] :\; (a,A)\in \N^{[<\infty]}\times \mathcal{H}\}$ is not,
in general, a basis for a topology on $\N^{[<\infty]}$, but the
following abstract version of the Baire property and related
concepts will be useful:

\begin{defn}
Let $\mathcal{H}\subset\N^{[\infty]}$ be a co-ideal.
$\mathcal{X}\subseteq\N^{[\infty]}$ has the abstract
$\mathcal{H}$-\textbf{Baire property} if for every
$[a,A]\neq\emptyset$ with $A\in\mathcal{H}$ there exists
$[b,B]\subseteq [a,A]$ with $B\in\mathcal{H}$ such that
$[b,B]\subseteq\mathcal{X}$ or $[b,B]\cap\mathcal{X} = \emptyset$.
$\mathcal{X}$ is $\mathcal{H}$-\textbf{nowhere dense} if for
every $[a,A]\neq\emptyset$ with $A\in\mathcal{H}$ there exists
$[b,B]\subseteq [a,A]$ with $B\in\mathcal{H}$ such that
$[b,B]\cap\mathcal{X} = \emptyset$. $\mathcal{X}$ is
$\mathcal{H}$-\textbf{meager} if it is the union of countably
many $\mathcal{H}$-nowhere dense sets.
\end{defn}

Given a decreasing sequence $A_0\supseteq A_1\supseteq
A_2\supseteq \cdots$ of infinite subsets of $\N$, a set $B$ is a
\textbf{diagonalization} of the sequence (or $B$
\textbf{diagonalizes} the sequence) if and only if $B/n\subseteq
A_n$ for each $n\in B$. A co-ideal $\mathcal{H}$ is
\textbf{selective} if and only if every decreasing sequence in
$\mathcal{H}$ has a diagonalization in $\mathcal{H}$.

\bigskip

A co-ideal $\mathcal H$ has the {\bf $Q^+$-property}, if for every
$A\in\coideal$ and every partition $(F_n)_n$ of $A$ into finite
sets, there is $S\in \coideal$ such that $S\subseteq A$ and
$|S\cap F_n|\leq 1$  for every $n\in\N$.

\begin{Prop}\cite{math}
A co-ideal $\mathcal{H}$ is selective if and only if the poset $(\mathcal{H}, \subseteq^*)$ is $\sigma$-closed
and  $\mathcal{H}$ has the $Q^+$-property.
\end{Prop}

Given a co-ideal $\mathcal{H}$, recall that a set $\mathcal
D\subseteq \mathcal H$ is \textbf{dense open} in the ordering
$(\mathcal{H},\subseteq)$, if (a) for every $A\in\mathcal H$ there
exists $B\in\mathcal D$ such that $B\subseteq A$ and, (b) for
every $A,B\in\mathcal H$, if $B\subseteq A$ and $A\in\mathcal D$
then $B\in\mathcal D$. Please notice that we will also consider
the ordering $(\mathcal{H},\subseteq^*)$, where $A\subseteq^* B$
if and only if $A\setminus B$ is a finite set, but any reference
to ``dense open'' in this paper will be \textit{only} with respect
to the ordering $(\mathcal{H},\subseteq)$.

\medskip

Given a sequence $\{D_n\}_{n\in\N}$ of
dense open sets in $(\mathcal{H},\subseteq)$, a set $B$ is a
\textbf{diagonalization} of $\{D_n\}_{n\in\N}$ if and only if
$B/n\in D_n$ for every $n\in B$. A co-ideal $\mathcal{H}$ is
\textbf{semiselective} if for every sequence $\{D_n\}_{n\in\N}$ of
dense open subsets of $\mathcal{H}$, the family of its
diagonalizations is dense in $(\mathcal{H},\subseteq)$.

\begin{Prop}\cite{far}
A co-ideal $\mathcal{H}$ is semiselective if and only if the poset
$(\mathcal{H}, \subseteq^*)$ is $\sigma$-distributive and
$\mathcal{H}$ has the $Q^+$-property.
\end{Prop}

Since $\sigma$-closedness implies $\sigma$-distributivity, then
semiselectivity follows from selectivity,  but the converse does
not hold (see \cite{far} for an example).

In section \ref{topchar} we list  results of Ellentuck, Mathias
and Farah that  characterize topologically the Ramsey property and
the local Ramsey property. In section \ref{main} we define a
family of games, and  present the main result, which states that a
co-ideal $\mathcal H$ is semiselective if and only if the $\mathcal
H$-Ramsey sets are exactly those for which the associated games
are determined. This generalizes results of Kastanas \cite{kas}
and Matet \cite{mate}. The proof is given in section
\ref{mainproof}. In section \ref{consistency} we show that in Solovay's model, for every semiselective co-ideal $\mathcal H$ all sets of real numbers from $L(\mathbb R)$ are $\mathcal H$-Ramsey.

\thanks{We thank A. Blass, J. Bagaria and the referee for helping us to
correct some deficiencies in previous versions of the article.}

\section{Topological characterization of the Ramsey property}\label{topchar}

The following are the main results concerning the characterization
of the Ramsey property and the local Ramsey property in
topological terms.

\begin{thm}\label{Ellentuck}
[Ellentuck] Let $\mathcal{X}\subseteq \N^{[\infty]}$ be given.

\begin{itemize}
\item[(i)] $\mathcal{X}$ is  Ramsey if and only if
$\mathcal{X}$ has the Baire property, with respect to Ellentuck's
topology.

\item[(ii)] $\mathcal{X}$ is Ramsey null if and only if
$\mathcal{X}$ is meager, with respect to Ellentuck's topology.
\end{itemize}
\end{thm}

\begin{thm}\label{Mathias}
[Mathias] Let $\mathcal{X}\subseteq \N^{[\infty]}$  and a selective co-ideal $\mathcal{H}$ be given.

\begin{itemize}

\item[(i)] $\mathcal{X}$ is $\mathcal{H}$-Ramsey if and only if
$\mathcal{X}$ has the abstract $\mathcal{H}$-Baire property.

\item[(ii)] $\mathcal{X}$ is $\mathcal{H}$-Ramsey null if and only
if $\mathcal{X}$ is $\mathcal{H}$-meager.
\end{itemize}
\end{thm}

\begin{thm}\label{Farah} [Farah, Todorcevic] Let $\mathcal{H}$ be a co-ideal. The following are
equivalent:

\begin{itemize}

\item[(i)] $\mathcal{H}$ is semiselective.

\item[(ii)] The $\mathcal{H}$-Ramsey subsets of $\N^{[\infty]}$
are exactly those sets having the abstract
$\mathcal{H}$-Baire property, and the following three families of subsets of
$\N^{[\infty]}$ coincide and are $\sigma$-ideals:
\begin{enumerate}

\item[(a)] $\mathcal{H}$-Ramsey null sets,

\item[(b)] $\mathcal{H}$-nowhere dense, and

\item[(c)] $\mathcal{H}$-meager sets.
\end{enumerate}
\end{itemize}
\end{thm}

In the next section we state results by Kastanas \cite{kas} and
Matet \cite{mate} (Theorems \ref{Kastanas} and \ref{Matet} below)
which are the game-theoretic counterparts of theorems
\ref{Ellentuck} and \ref{Mathias}, respectively; and we also
present our main result (Theorem \ref{Main} below), which is the
game-theoretic counterpart of Theorem \ref{Farah}.

\section{Characterizing the Ramsey property with games}\label{main}

The following is a relativized version of a game due to Kastanas  \cite{kas},
employed to obtain a characterization of the family
of completely Ramsey sets (i.e. $\mathcal{H}$-Ramsey for
$\mathcal{H}=\N^{[\infty]}$). The same game was used by Matet in
\cite{mate} to obtain the analog result when $\mathcal{H}$ is
selective.

\bigskip

Let $\mathcal{H}\subseteq \N^{[\infty]}$ be a fixed co-ideal. For
each $\mathcal{X} \subseteq \N^{[\infty]}$, $A\in \mathcal{H}$ and
$a\in \N^{[<\infty]}$ we define a two-player game
$G_{\mathcal{H}}(a,A,\mathcal{X})$ as follows: player I chooses an
element $A_{0}\in \mathcal{H}\upharpoonright A$; II answers by
playing $n_{0}\in A_{0}$ such that $a\subseteq n_{0}$, and
$B_{0}\in \mathcal{H}\cap (A_{0}/n_{0})^{[\infty]}$; then I
chooses $A_{1}\in \mathcal{H}\cap B_{0}^{[\infty]}$; II answers by
playing $n_{1}\in A_{1}$  and $B_{1}\in \mathcal{H}\cap
(A_{1}/n_{1})^{[\infty]}$; and so on. Player I wins if and only if
$a\cup \{n_{j}:\; j\in \N\}\in \mathcal{X}$.

\bigskip

$
\begin{array}{lcccccccccc}
I & A_0 &         & A_1 &        & \cdots & A_k &  & \cdots\\
\\
II &   & n_0, B_0 &     &n_1, B_1&\cdots  &    & n_k, B_k&\cdots

\end{array}
$

\bigskip

A \textbf{strategy} for a player is a rule that tells him (or her) what to
play based on the previous moves. A strategy is  a \textbf{winning
strategy for player I}  if player I wins the game whenever she (or he)
follows the strategy, no matter what player II plays. Analogously,
it can be defined what is a winning strategy for player II. The
precise definitions of strategy for two players games can be found
in \cite{kec,mos}.

\bigskip

Let $s = \{s_0, \dots, s_k\}$ be a nonempty finite subset of $\N$,
written in its increasing order, and $\overrightarrow{B} = \{B_0,
\dots, B_k\}$ be a sequence of elements of $\mathcal{H}$. We say
that the pair $(s,\overrightarrow{B})$ is a \textbf{legal position
for player II} if $(s_0, B_0), \dots, (s_k, B_k)$ is a sequence of
possible consecutive moves of II in the game
$G_{\mathcal{H}}(a,A,\mathcal{X})$, respecting the rules. In this
case, if $\sigma$ is a winning strategy for player I in the game,
we say that $\sigma(s,\overrightarrow{B})$ is a \textbf{realizable
move of player I} according to $\sigma$. Notice that if $r\in
B_k/s_k$ and $C\in\mathcal{H}\upharpoonright B_k/s_k$ then $(s_0,
B_0), \dots, (s_k, B_k), (r,C)$ is also a sequence of possible
consecutive moves of II in the game. We will sometimes use the
notation $(s,\overrightarrow{B}, r, C)$, and say that
$(s,\overrightarrow{B}, r, C)$ is a legal position for player II
and $\sigma(s,\overrightarrow{B}, r, C)$ is a realizable move of
player I according to $\sigma$.

\bigskip

We say that the game $G_{\mathcal{H}}(a,A,\mathcal{X})$ is
\textbf{determined} if one of the players has a winning strategy.

\bigskip

\begin{thm} \label{Kastanas}
[Kastanas] $\mathcal{X}$ is Ramsey if and only if for every
$A\in \N^{[\infty]}$ and $a\in \N^{[<\infty]}$ the game
 $G_{\N^{[\infty]}}(a,A,\mathcal{X})$ is determined.
\end{thm}

\begin{thm}\label{Matet}
[Matet] Let $\mathcal{H}$ be a selective co-ideal.
$\mathcal{X}$ is $\mathcal{H}$-Ramsey if and only if for every
$A\in \mathcal{H}$ and $a\in \N^{[<\infty]}$ the game
 $G_{\mathcal{H}}(a,A,\mathcal{X})$ is determined.
\end{thm}

Now we state our main result:

\begin{thm}\label{Main}
Let $\mathcal{H}$ be a co-ideal. The following are
equivalent:
\begin{enumerate}
\item{} $\mathcal{H}$ is semiselective.

\item{} $\forall\mathcal{X}\subseteq\N^{[\infty]}$,
$\mathcal{X}$ is $\mathcal{H}$-Ramsey if and only if for every
$A\in \mathcal{H}$ and $a\in \N^{[<\infty]}$ the game
 $G_{\mathcal{H}}(a,A,\mathcal{X})$ is determined.
 \end{enumerate}

\end{thm}

So Theorem \ref{Main} is a game-theoretic counterpart to Theorem \ref{Farah} in the previous
section, in the sense that it gives us a game-theoretic
characterization of semiselectivity. Obviously, it also gives us a
characterization of the $\mathcal{H}$-Ramsey property, for semiselective
$\mathcal{H}$, which generalizes the main results of Kastanas in
\cite{kas}  and  Matet in \cite{mate} (Theorems \ref{Kastanas} and
\ref{Matet} above).

It is known that every analytic set is $\mathcal{H}$-Ramsey for
$\mathcal H$ semiselective (see Theorem 2.2 in \cite{far} or Lemma
7.18 in \cite{todo}). Assuming $AD_{\mathbb R}$, i.e., the axiom
of determinacy for games over the reals (see \cite{kec} or
\cite{mos}), we obtain the following from Theorem \ref{Main}:

\begin{Coro}
\label{projective sets} Assume $AD_{\mathbb R}$. If $\mathcal H$
is a semiselective co-ideal then every subset of $\N^{[\infty]}$
is $\mathcal H$-Ramsey.
\end{Coro}

\qed

\section{Proof of the main result}\label{mainproof}

Throughout the rest of this section, fix a semiselective co-ideal
$\mathcal{H}$. Before proving Theorem \ref{Main}, in Propositions
\ref{playerI_semiselective} and \ref{playerII_semiselective} below
we will deal with winning strategies of players in a game
$G_{\mathcal{H}}(a,A,\mathcal{X})$.

\begin{Prop}\label{playerI_semiselective}
For every $\mathcal{X}\subseteq\N^{[\infty]}$, $A\in \mathcal{H}$ and $a\in
\N^{[<\infty]}$, I has a winning strategy in
$G_{\mathcal{H}}(a,A,\mathcal{X})$ if and only if there exists $E\in
\mathcal{H}\upharpoonright A$ such that $[a,E]\subseteq
\mathcal{X}$.
\end{Prop}

\begin{proof}
Suppose $\sigma$ is a winning strategy for I. We will suppose that
$a = \emptyset$ and $A = \mathbb{N}$ without loss of generality.

\medskip

Let $A_{0} = \sigma(\emptyset)$ be the first move of I using
$\sigma$. We will define a tree $T$  of finite subsets of $A_0$;
and for each $s\in T$ we will also define a family $M_s\subseteq
A_0^{[\infty]}$ and a family $N_s\subseteq
(A_0^{[\infty]})^{\lvert s\rvert}$, where $\lvert s\rvert$ is the
length of $s$. Put $\{p\}\in T$ for each $p\in A_0$ and let

$$
M_{\{ p\}} \subseteq \{\sigma(p,B) :\; B\in
\mathcal{H}\upharpoonright A_0\}
$$
be a maximal $\mathcal{H}$-disjoint family (see  paragraph after Definition \ref{co-ideal}), and set
$$
N_{\{ p\} } = \{\{B\} :\; \sigma(p,B)\in M_{\{ p\}}\}.
$$

\medskip

Suppose we have defined $T\cap  A_0^{[n]}$ and we have chosen a
maximal $\mathcal{H}$-disjoint family $M_s$ of realizable moves of
player I of the form $\sigma(s,\overrightarrow{B})$ for every
$s\in T\cap A_0^{[n]}$ . Let
\[
N_s = \{\overrightarrow{B} :\; \sigma(s,\overrightarrow{B})\in
M_s\}.
\]
Given  $s\in T\cap A_0^{[n]}$, $\overrightarrow{B}\in N_s$ and
$r\in \sigma(s,\overrightarrow{B})/s$, we put $s\cup\{r\}\in T$.
Then choose a maximal $\mathcal{H}$-disjoint family
$$
M_{s\cup\{r\}} \subseteq \{\sigma(s,\overrightarrow{B},r,C):\;
\overrightarrow{B}\in N_s,\;
C\in\mathcal{H}\upharpoonright\sigma(s,\overrightarrow{B})/r\}.
$$
Put
$$
N_{s\cup\{r\}} = \{(\overrightarrow{B},C) :
\;\sigma(s,\overrightarrow{B},r,C)\in M_{s\cup\{r\}}\}.
$$

\medskip

Now, for every $s\in T$, let  $$\mathcal{U}_s = \{E\in\mathcal{H}
: (\exists F\in M_s)\ E\subseteq F\}\ \ \mbox{and}$$
$$\mathcal{V}_s = \{E\in\mathcal{H} : (\forall F\in
M_{s\setminus\{max(s)\}})\ max(s)\in F\ \rightarrow F\cap
E\not\in\mathcal{H}\}.$$

\bigskip

\begin{clm}\label{dense_moves}
For every $s\in T$, $\mathcal{U}_s\cup\mathcal{V}_s$ is dense open
in $(\mathcal{H}\upharpoonright A_0,\subseteq)$.
\end{clm}

\begin{proof}
Fix $s\in T$ and $A\in\mathcal{H}\upharpoonright A_0$. If
$(\forall F\in M_{s\setminus\{max(s)\}})\ max(s)\in F\ \rightarrow
F\cap A\not\in\mathcal{H}$ holds, then $A\in\mathcal{V}_s$.
Otherwise, fix $F\in M_{s\setminus\{max(s)\}}$ such that
$max(s)\in F$ and $F\cap A \in\mathcal{H}$. Let
$\overrightarrow{B}\in N_{s\setminus\{max(s)\}}$ be such that
$\sigma(s\setminus\{max(s)\}, \overrightarrow{B}) = F$. Notice that
since $max(s)\in F$ then $$(s\setminus\{max(s)\},
\overrightarrow{B}, max(s), F\cap A/max(s))$$ is a legal position
for player II. Then, using the maximality of $M_s$, choose
$\hat{F}\in M_s$ such that $$E : = \sigma(s\setminus\{max(s),
\overrightarrow{B}, max(s), F\cap A/max(s))\cap\hat{F}$$ is in
$\mathcal{H}$. So $E\in\mathcal{U}_s$ and $E\subseteq A$. This
completes the proof of claim \ref{dense_moves}.
\end{proof}

\begin{clm}\label{Compatible_Moves_dense}
There exists $E\in\mathcal{H}\upharpoonright A_0$ such that for
every $s\in T$ with $s\subset E$, $E/s \in \mathcal{U}_s$.
\end{clm}

\begin{proof}
For each $n\in \N$ , let
$$
\begin{array}{lcl}
\mathcal{D}_n  & = &  \bigcap_{max(s) =
n}\mathcal{U}_s\cup\mathcal{V}_s.
\\
\\
\mathcal{U}_n  & = &  \bigcap_{max(s) = n}\mathcal{U}_s\ \mbox{,}
\end{array}
$$
(if there is no $s\in T$ with $max(s)=n$, then we put
$\mathcal{D}_n = \mathcal{U}_n =\mathcal{H}\upharpoonright
A_0$). By Claim \ref{dense_moves}, every
$\mathcal{D}_n$ is dense open in $(\mathcal{H}\upharpoonright
A_0,\subseteq)$. Using semiselectivity, choose a diagonalization
$\hat{E}\in\mathcal{H}\upharpoonright A_0$ of the sequence
$(\mathcal{D}_n)_n$. Let

$$E_0 : = \{n\in \hat{E} : \hat{E}/n\in\mathcal{U}_n\}\ \ \mbox{and}\ \ E_1 := \hat{E}\setminus E_0.$$

\medskip

Let us prove that $E_1\not\in\mathcal{H}$:

\medskip

Suppose $E_1\in\mathcal{H}$. By the definitions, $(\forall n\in
E_1)\ \hat{E}/n\not\in \mathcal{U}_n$. Let $n_0 = min(E_1)$ and
fix $s_0\subset \hat{E}$ such that $max(s_0) = n_0$ and
satisfying, in particular, the following:  $$(\forall F\in
M_{s_0\setminus\{n_0\}})\ n_0\in F\ \rightarrow\ F\cap
E_1/n_0\not\in\mathcal{H}.$$ Notice that $|s_0|>1$, by the
construction of the $M_s$'s.

Now, let $m = max(s_0\setminus\{n_0\})$. Then $m\in E_0$ and
therefore
$\hat{E}/m\in\mathcal{U}_m\subseteq\mathcal{U}_{s_0\setminus\{n_0\}}$.
So there exists $F\in M_{s_0\setminus\{n_0\}}$ such that
$\hat{E}/m \subseteq F$. Since $m<n_0$ then $n_0\in F$. But $F\cap
E_1/n_0 = E_1/n_0\in\mathcal{H}$. A contradiction.

\medskip
Hence, $E_1\not\in\mathcal{H}$ and therefore $E_0\in\mathcal{H}$. Then $E : = E_0$ is as required.
\end{proof}

\begin{clm}\label{Compatible_Moves}
Let $E$ be as in Claim \ref{Compatible_Moves_dense} and $s\cup \{r\}\in T$ with $s\subset E$ and $r\in E/s$. If
$E/s\subseteq\sigma(s,\overrightarrow{B})$ for some
$\overrightarrow{B}\in N_s$, then there exists
$C\in\mathcal{H}\upharpoonright\sigma(s,\overrightarrow{B})/r$
such that $E/r\subseteq \sigma(s,\overrightarrow{B},r,C)$ and
$(\overrightarrow{B},C)\in N_{s\cup\{r\}}$.
\end{clm}

\begin{proof}
Fix  $s$ and $r$ as in the hypothesis. Suppose $E/s\subseteq
\sigma(s,\overrightarrow{B})$ for some $\overrightarrow{B}\in
N_s$. Since $E/r\in\mathcal{U}_{s\cup\{r\}}$, there exists
$(\overrightarrow{D}, C)\in N_{s\cup\{r\}}$ such that
$E/r\subseteq \sigma(s,\overrightarrow{D},r,C)$. Notice that
$E/r\subseteq\sigma(s,\overrightarrow{B})\cap\sigma(s,\overrightarrow{D})$.
Since $M_{s}$ is $\mathcal{H}$-disjoint, then
$\sigma(s,\overrightarrow{D})$ is neccesarily equal to
$\sigma(s,\overrightarrow{B})$ and therefore
$\sigma(s,\overrightarrow{B},r,C) =
\sigma(s,\overrightarrow{D},r,C)$. Hence $(\overrightarrow{B},
C)\in N_{s\cup\{r\}}$ and $E/r\subseteq
\sigma(s,\overrightarrow{B},r,C)$.
\end{proof}

\begin{clm}\label{ganadora}
Let $E$ be as in Claim \ref{Compatible_Moves_dense}. Then $[\emptyset,E]\subseteq\mathcal{X}$.
\end{clm}

\begin{proof}
Let $\{k_i\}_{i\geq 0}\subseteq E$ be given. Since $E/k_0
\in\mathcal{U}_{\{k_0\}}$, there exists $B_0\in
N_{\{k_0\}}$ such that $E/k_0\subseteq\sigma(k_0,B_0)$. Thus, by the choice of $E$ and
applying Claim \ref{Compatible_Moves} iteratively, we prove that $\{k_i\}_{i\geq
0}$ is generated in a run of the game in which player I has used his winning strategy
$\sigma$. Therefore $\{k_i\}_{i\geq 0}\in\mathcal{X}$.
\end{proof}

 The converse is trivial. This completes the proof of Proposition \ref{playerI_semiselective}.
\end{proof}

\bigskip

Now we turn to the case when player II has a winning strategy. The proof of the following is similar to
the proof of Proposition 4.3 in \cite{mate}. First we show a
result we will need in the sequel, it should be compared with
lemma 4.2 in \cite{mate}.

\begin{lem}\label{lemmaPlayerII}
Let $B\in \mathcal{H}$, $f:\mathcal{H}\upharpoonright B \rightarrow \N$, and
$g: \mathcal{H}\upharpoonright B \rightarrow
\mathcal{H}\upharpoonright B$ be given such that $f(A)\in A$ and
$g(A)\subseteq A/f(A)$. Then there is $E_{f,g}\in
\mathcal{H}\upharpoonright B$ with the property that for each
$p\in E_{f,g}$ there exists $A\in \mathcal{H}\upharpoonright B$
such that $f(A)=p$ and $E_{f,g}/p\subseteq g(A)$.
\end{lem}

\begin{proof}
For each $n\in \{f(A) : A\in\mathcal{H}\upharpoonright B\}$, let

\medskip

$$
U_n = \{E\in \mathcal{H}\upharpoonright B :\; (\exists A\in
\mathcal{H}\upharpoonright B)\ (f(A) = n \wedge E\subseteq g(A))\}$$

and

$$
V_n = \{E\in \mathcal{H}\upharpoonright B :\; (\forall A\in
\mathcal{H}\upharpoonright B)\ (f(A) = n\ \rightarrow\ \mid
g(A)\setminus E \mid = \infty)\}.
$$

The set $D_n = U_n\cup V_n$ is dense open in
$\mathcal{H}\upharpoonright B$. Choose $E\in
\mathcal{H}\upharpoonright B$ such that for each $n\in E$, $E/n\in
D_n$. Let
$$
E_{0}=\{n\in E :\; E/n\in U_n\} \ \mbox{and} \ E_{1}=\{n\in E
:\; E/n\in V_n\}.
$$
Now, suppose $E_1\in \mathcal{H}$. Then, for each $n\in E_1$,
$E_1/n\in V_n$. Let $n_1= f(E_1)$. So $n_1\in E_1$ by
the definition of $f$. But, by the definition of $g$,
$g(E_1)\subseteq E_1/n_1$ and so $E_1/n_1\not\in
V_{n_1}$; a contradiction. Therefore, $E_{1}\not\in \mathcal{H}$.
Hence $E_{0}\in \mathcal{H}$, since $\mathcal{H}$ is a co-ideal.
The set $E_{f,g} : = E_{0}$ is as required.
\end{proof}

\begin{Prop}\label{playerII_semiselective}
For every $\mathcal{X} \subseteq \N^{[\infty]}$, $A\in\mathcal{H}$ and $a\in
\N^{[<\infty]}$, II has a winning strategy in
$G_{\mathcal{H}}(a,A,\mathcal{X})$ if and only if $\forall
A'\in\mathcal{H}\upharpoonright A$ there exists
$E\in\mathcal{H}\upharpoonright A'$ such that $[a,E]\cap\mathcal{X}
= \emptyset$.

\end{Prop}
\begin{proof}

Let $\tau$ be a winning strategy for II in $G_{\mathcal{H}}(a,A,\mathcal{X})$
and let $A'\in\mathcal{H}\upharpoonright A$ be given. We are going to define
a winning strategy $\sigma$ for I, in $G_{\mathcal{H}}(a,A',\N^{[\infty]}\setminus
\mathcal{X})$, in such a way that we will get the required result by
means of Proposition \ref{playerI_semiselective}. So, in a play of the game
$G_{\mathcal{H}}(a,A',\N^{[\infty]}\setminus \mathcal{X})$, with II's
successive moves being $(n_j,B_j)$, $j\in\N$, define
$A_j\in \mathcal{H}$ and $E_{f_j,g_j}$ as in Lemma \ref{lemmaPlayerII}, for $f_j$ and $g_j$ such that

\begin{itemize}
\item[{(1)}] For all $\hat{A}\in \mathcal{H}\upharpoonright A'$,
\[
(f_0(\hat{A}),g_0(\hat{A}))= \tau(\hat{A});
\]

\item[{(2)}] For all $\hat{A}\in \mathcal{H}\upharpoonright
B_j\cap g_j(A_j)$,
\[
(f_{j+1}(\hat{A}),g_{j+1}(\hat{A}))=
\tau(A_{0},\cdots,A_j,\hat{A});
\]

\item[{(3)}] $A_{0}\subseteq A'$, and $A_{j+1}\subseteq B_j\cap
g_j(A_j)$;

\item[{(4)}] $n_j=f_j(A_j)$ and $E_{f_j,g_j}/n_j\subseteq
g_j(A_j)$.
\end{itemize}

Now, let $\sigma (\emptyset)$=$E_{f_0,g_0}$ and $\sigma
((n_0,B_0),\cdots ,(n_j,B_j))$=$E_{f_{j+1},g_{j+1}}$.

\bigskip

Conversely, let $A_0$ be the first move of I in the game. Then
there exists $E\in \mathcal{H}\upharpoonright A_0$ such that
$[a,E]\cap \mathcal{X} = \emptyset$. We define a winning strategy
for player II by letting her (or him) play $(\min E,E\setminus
\{\min E \})$ at the first turn, and arbitrarily from there on.
\end{proof}

\bigskip

We are ready now for the following:

\bigskip

\begin{proof}[Proof of Theorem \ref{Main}]
If $\mathcal{H}$ is semiselective, then part 2 of Theorem
\ref{Main} follows from Propositions \ref{playerI_semiselective}
and \ref{playerII_semiselective}.

\medskip

Conversely, suppose part 2 holds and let $(\mathcal{D}_n)_n$ be a
sequence of dense open sets in $(\mathcal{H},\subseteq)$. For
every $a\in\N^{[<\infty]}$, let
$$
\mathcal{X}_a = \{B\in [a,\N] : B/a\ \mbox{diagonalizes some decreasing}\ (A_n)_n\ \mbox{such that}\ (\forall n)\ A_n\in \mathcal{D}_n\}
$$
and define $$\mathcal{X} = \bigcup_{a\in\N^{[<\infty]}} \mathcal{X}_a.$$

Fix $A\in \mathcal{H}$ and $a\in \N^{[<\infty]}$ with
$[a,A]\neq\emptyset$, and define a winning strategy $\sigma$ for
player I in $G_{\mathcal{H}}(a,A,\mathcal{X})$, as follows: let
$\sigma(\emptyset)$ be any element of $\mathcal{D}_0$ such that
$\sigma(\emptyset)\subseteq A$. At stage $k$, if II's successive
moves in the game are $(n_j,B_j)$, $j\leq k$, let
$\sigma((n_0,B_0), \dots, (n_k,B_k))$ be any element of
$\mathcal{D}_{k+1}$ such that $\sigma((n_0,B_0), \dots,
(n_k,B_k))\subseteq B_k$. Notice that $a\cup\{n_0, n_1, n_2,
\dots\}\in\mathcal{X}_a$.

So the game $G_{\mathcal{H}}(a,A,\mathcal{X})$ is determined for
every $A\in \mathcal{H}$ and $a\in \N^{[<\infty]}$ with
$[a,A]\neq\emptyset$. Then, by our assumptions, $\mathcal{X}$ is
$\mathcal{H}$-Ramsey. So given $A\in\mathcal{H}$, there exists
$B\in\mathcal{H}\upharpoonright A$ such that
$B^{[\infty]}\subseteq\mathcal{X}$ or $B^{[\infty]}\cap\mathcal{X}
= \emptyset$. The second alternative does not hold, so
$\mathcal{X}\cap\mathcal{H}$ is dense in $(\mathcal{H},
\subseteq)$. Hence, $\mathcal{H}$ is semiselective.
\end{proof}

\section{The Ramsey property in Solovay models}\label{consistency}

Recall that the Mathias forcing notion $\mathbb M$  is the collection of all the sets of the form
\[
[a,A] : = \{B\in\N^{[\infty]} : a\sqsubset B\subseteq A\},
\]
ordered by
$[a,A]\leq [b,B]$ if and only if  $[a,A]\subseteq [b,B]$.

If $\mathcal H$ is a co-ideal, then $\mah$, the Mathias partial
order with respect to $\mathcal H$ is the collection of all the
$[a,A]$ as above but with $A\in \mathcal H$, ordered in the same
way.

A co-ideal $\mathcal H$ has the  \textit{Mathias property} if it
satisfies that if $x$ is  $\mah$-generic over a model $M$,  then
every $y\in x^{[\infty]}$ is $\mah$-generic over $M$. And
$\mathcal H$ has the \textit{Prikry property} if for every
$[a,A]\in\mah$ and every formula $\varphi$ of the forcing language
of  $\mah$, there is $B\in\mathcal H \upharpoonright A$ such that
$[a,B]$ decides $\varphi$.

\begin{thm}{(\cite{far}, Theorem 4.1)}
For a co-ideal $\mathcal H$ the following are equivalent.
\begin{enumerate}
\item $\mathcal H$ is semiselective.
\item $\mah$ has the Prikry property.
\item $\mah$ has the Mathias property.
\end{enumerate}
\end{thm}

Suppose $M$ is a model of $ZFC$ and there is a inaccessible cardinal
$\lambda$  in $M$. The Levy partial order  $\col$ produces a generic
extension $M[G]$ of $M$ where $\lambda$ becomes $\aleph_1$.
Solovay's model (see \cite{sol}) is obtained by taking the submodel of $M[G]$ formed
by all the sets hereditarily definable  in $M[G]$ from a sequence of
ordinals (see \cite{math}, or \cite{je}).

In \cite{math}, Mathias shows that if $V=L$,  $\lambda$ is a Mahlo
cardinal, and $V[G]$ is a generic extension obtained by forcing
with $\col$, then every  set of real numbers defined  in the
generic extension from a sequence of ordinals is $\mathcal
H$-Ramsey for $\mathcal H$ a selective co-ideal. This result can
be extended to semiselective co-ideals under suitable large
cardinal hypothesis.

\begin{thm}\label{wchramsey}
Suppose  $\lambda$ is a weakly compact cardinal. Let $V[G]$ be a
generic extension by $\col$. Then, if $\mathcal H$ is a
semiselective co-ideal in $V[G]$, every set of real numbers in
$L(\mathbb R)$ of $V[G]$  is $\mathcal H$-Ramsey.
\end{thm}

\begin{proof}
Let $\mathcal H$ be a semiselective co-ideal in $V[G]$. Let $\mathcal
A$ be a set of reals in $L(\mathbb R)^{V[G]}$; in particular,
$\mathcal A$ is defined in $V[G]$ by a formula $\varphi$ from a
sequence of ordinals. Let $[a,A]$ be a condition of the Mathias
forcing $\mathbb M_{\mathcal H}$ with respect to the semiselective
co-ideal $\mathcal H$. Let finally $\dot{\mathcal H}$ be a name for
$\mathcal H$. Notice that $\dot{\mathcal H}\subseteq V_\lambda$.

Since $V[G]$ satisfies that $\mathcal H$ is semiselective, the
following statement holds in $V[G]$: For every sequence $D=(D_n:
n\in \omega)$ of open dense subsets of $\mathcal H$ and for every
$x\in \mathcal H$ there is $y\in \mathcal H$, $y\subseteq x$, such
that $y$ diagonalizes the sequence $D$.

Therefore, there is $p\in G$ such that, in $V$, the following
statement holds:

\[\begin{aligned}
\notag \forall \dot D \forall \tau ( p\Vdash_{\col} (\dot D \mbox{
is a name for a sequence of dense open subsets of } \dot{\mathcal H}
\\\notag \mbox{ and }   \tau\in \dot{\mathcal H}) \longrightarrow
(\exists x (x \in \dot{\mathcal H}, x\subseteq \tau ,
x \mbox{ diagonalizes } \dot D))) .
\end{aligned}\]

Notice that every real in $V[G]$ has a name in $V_\lambda$, and
names for subsets of $\mathcal H$ or countable sequences of subsets
of $\mathcal H$ are contained in $V_\lambda$. Also, the forcing
$\col$ is a subset of $V_\lambda$. Therefore the same statement is
valid in the structure $(V_\lambda, \in, \dot{\mathcal H}, \col)$.
This statement is $\Pi^1_1$ over this structure, and since $\lambda$
is $\Pi^1_1$-indescribable, there is $\kappa<\lambda$ such that in
$(V_\kappa, \in ,\dot{\mathcal H}\cap V_\kappa, \col\cap V_\kappa)$
it holds

\[\begin{aligned}\notag
\forall \dot D \forall \tau (p\Vdash_{Col(\omega, <\kappa)} (\dot D
\mbox{ is a name for a sequence of dense open subsets of }
\dot{\mathcal H}\cap V_\kappa \\\notag\mbox{ and }  \tau\in
\dot{\mathcal H}\cap V_\kappa) \longrightarrow (\exists x
(x\in \dot{\mathcal H}\cap V_\kappa, x\subseteq \tau ,
x \mbox{ diagonalizes } \dot D))).
\end{aligned}\]

We can get  $\kappa$  inaccessible, since there is a $\Pi^1_1$
formula expressing  that $\lambda$ is inaccessible. Also, $\kappa$
is such that $p$ and the names for the real parameters in the
definition of $\mathcal A$ and for $A$ belong to $V_\kappa$.

If we let $G_\kappa = G\cap Col(\omega, <\kappa)$, then
$G_\kappa\subseteq Col(\omega, <\kappa)$,  and is generic over $V$.
Also, $p\in G_\kappa$. $\dot{\mathcal H}\cap V_\kappa$ is a
$Col(\omega, <\kappa)$-name in $V$ which is interpreted by
$G_\kappa$ as $\mathcal H\cap V[G_\kappa]$, thus $\mathcal H\cap
V[G_\kappa]\in V[G_\kappa]$. And since every subset (or sequence of
subsets) of $\mathcal H\cap V[G_\kappa]$ which belongs to
$V[G_\kappa]$ has a name contained in $V_\kappa$, we have that, in
$V[G_\kappa]$, $\mathcal H\cap V_\kappa$ is semiselective,  and in
consequence it has both the Prikry and the Mathias properties.

Now the proof can be finished as in \cite{math}. Let $\dot r$ be the
canonical name of a $\mathbb M_{\mathcal H \cap V[G_\kappa]}$
generic real, and consider the formula $ \varphi( \dot r)$ in the
forcing language of $V[G_\kappa]$. By the Prikry property of
$\mathcal H \cap V[G_\kappa]$, there is $A'\subseteq A$, $A'\in
\mathcal H \cap V[G_\kappa]$, such that $[a,A']$ decides $\varphi(
\dot r)$. Since $2^{2^\omega}$ computed in $V[G_\kappa] $ is
countable in $V[G]$, there is (in $V[G]$) a  $\mathbb M_{\mathcal H
\cap V[G_\kappa]}$-generic  real $x$ over $V[G_\kappa]$ such that
$x\in [a,A']$. To see that there is such a generic real in $\mathcal
H$ we argue as in 5.5 of \cite{math}  using the semiselectivity of
$\mathcal H$ and the fact that $\mathcal H\cap V[G_\kappa]$ is
countable in $V[G]$ to obtain an element of $\mathcal H$ which is
generic.
 By the Mathias property
of $\mathcal H \cap V[G_\kappa]$, every  $y\in [a,x\setminus a]$ is
also $\mathbb M_{\mathcal H \cap V[G_\kappa]}$-generic over
$V[G_\kappa]$, and also $y\in [a,A']$. Thus $\varphi(x)$ if and only
if $[a,A']\Vdash\varphi(\dot r)$, if and only if
 $\varphi (y)$.
Therefore, $[a,x\setminus a]$ is contained in $\mathcal A$ or is
disjoint from $\mathcal A$.

\end{proof}

As in \cite{math}, we obtain the following.

\begin{Coro}
If $ZFC$ is consistent with the existence of a weakly compact
cardinal, then so is the statement that for every semiselective
co-ideal $\mathcal H$ all sets of real numbers from $L(\mathbb R)$
are $\mathcal H$-Ramsey.
\end{Coro}

Eisworth (\cite{ei}) showed that the hypothesis of the existence
of a Mahlo cardinal in Mathias� result cannot be weakened.

{\bf Question:} Can the weakly compact cardinal hypothesis in the
statement of Theorem \ref{wchramsey} be weakened? Would a Mahlo
cardinal suffice?

\end{document}